\documentclass[10pt,a4paper,final]{amsart} 
\usepackage[left=23mm,right=23mm,top=45mm,bottom=50mm]{geometry} 

\usepackage{amsmath,amsthm,amssymb,latexsym}
\usepackage{enumerate}
\usepackage{pifont,wasysym}
\usepackage[T1]{fontenc}
\usepackage{color}

\usepackage{shuffle}

\usepackage{hyperref} 

\theoremstyle{plain} 
\newtheorem{theorem}{Theorem}
\newtheorem{proposition}{Proposition}
\newtheorem{corollary}{Corollary}[section]
\newtheorem{lemma}[corollary]{Lemma}

\theoremstyle{remark}

\newcommand{\NN}{\mathbb{N}}

\newcommand{\QQ}{\mathbb{Q}}
\newcommand{\RR}{\mathbb{R}}

\newcommand{\Id}{\operatorname{Id}}

\newcommand{\Set}[1]{\left\{ #1  \right\}}
\newcommand{\SetSuchThat}[2]{\left\{\, #1 \ | \ #2 \, \right\}}

\newcommand{\Exp}[1]{e^{#1}}
\newcommand{\ad}[1]{\operatorname{ad}_{#1}}

\newcommand{\Bra}[1]{\left(#1\right)}

\newcommand{\Ring}{\mathbf{k}}
\newcommand{\Algebra}{R}

\newcommand{\RotaOp}{P}
\newcommand{\RotaOpT}{\tilde P}
\newcommand{\RotaOpBra}[1]{\RotaOp\left(#1\right)}
\newcommand{\RotaOpTBra}[1]{\RotaOpT\left(#1\right)}

\newcommand{\QInt}{P_q}
\newcommand{\QIntTwo}{{\bar P}_q}

\newcommand{\CHL}{\chi_\lambda}
\newcommand{\CHZ}{\chi_0}
\newcommand{\BCH}{{\rm BCH}}

\newcommand{\Functions}{C^\omega (\RR)}
\newcommand{\RotaInt}{P_{\int}}

\begin{document}

\title{Inhomogeneous linear equation in Rota-Baxter algebra}

\author{Gabriel Pietrzkowski}

\address{ University of Warsaw, Banacha 2, 02-097 Warsaw, Poland} 



\begin{abstract}
We consider a complete filtered Rota-Baxter algebra of weight $\lambda$ over a commutative ring. Finding the unique solution of a non-homogeneous linear algebraic equation in this algebra, we generalize Spitzer's identity in both commutative and non-commutative cases. As an application, considering the Rota-Baxter algebra of power series in one variable with q-integral as the Rota-Baxter operator, we show certain Eulerian identities. 
\end{abstract}

\maketitle

\section{Introduction}
\label{sec:Introduction}

Let $\Ring$ be a commutative ring and $\Algebra$ be a $\Ring$-algebra. For a fixed $\lambda\in\Ring$ assume we have a linear operator $\RotaOp : \Algebra \to \Algebra$ satisfying
\begin{align}
\label{eq:BaxterFormula}
 \RotaOp(x)\RotaOp(y) = \RotaOp(x\RotaOp(y)) + \RotaOp(\RotaOp(x)y) + \lambda \RotaOp(xy)
\end{align}
for all $x,y\in\Algebra$. Then $(\Algebra,\RotaOp)$ is a Rota-Baxter algebra of weight $\lambda$, and $\RotaOp$ is called a Rota-Baxter operator of weight $\lambda$. There are plenty of examples of Rota-Baxter algebras in different mathematical areas. The $\RR$-algebra of analytic functions $\Functions$ on the real line, with the integral operator is a R-B algebra of weight $0$ (see section \ref{sec:Result} for more details). Also  the algebra of linear operators with the integral operator is a non-commutative R-B algebra of weight $0$.  The algebra of formal power series in variable $t$ with rational coefficients $\QQ[[t]]$, with  the $q$-integral is a Rota-Baxter algebra of weight $1$ (see section \ref{sec:Eulerian})\cite{rota1969baxter,  guo2012introduction}. The algebra of sequences $(a_n)$ with values in $\Ring$ and the partial sum operator $\RotaOp(a_1, a_2,\cdots) = (a_1, a_1 + a_2, a_1 + a_2 + a_3,\cdots)$ is a Rota-Baxter algebra of weight $-1$ \cite{baxter1960analytic,rota1969baxter, 
guo2012introduction}. The algebra of functions $\varphi(t) = \int_{-\infty}^\infty e^{itx} dF(x)$, where $F:\RR\to\RR$ is a function of 
bounded variation such that $\lim_{x\to - \infty}F(x) = F(-\infty)$ exists, with an operator $\RotaOp(\varphi)(t) = \int_{0}^\infty e^{itx} dF(x) + F(0) - F(-\infty)$ is a Rota-Baxter algebra of weight $-1$ \cite{rota1972fluctuation,guo2012introduction}. Many other examples can be found in \cite{rota1969baxter,rota1972fluctuation,guo2012introduction,rota1995baxter}. 

Among many interesting results concerning Rota-Baxter algebras (see a monograph \cite{guo2012introduction}) there is a founding one, Spitzer's identity
\begin{align*}
 \exp \left(\RotaOp\left(\lambda^{-1}\log(1+\lambda a) \right) \right) =\sum_{n=0}^\infty \underbrace{\RotaOp(a\cdots\RotaOp(a}_{n})\cdots),
\end{align*}
where $a\in\Algebra$.
The name comes after Frank Spitzer who in 1956 gave expression for the characteristic function of a class of random variables using combinatorial tools \cite{spitzer1956combinatorial}.  Four years later Glen Baxter \cite{baxter1960analytic} realized that the identity given by Spitzer can by obtained using integral operator satisfying  \eqref{eq:BaxterFormula}. Actually he proved the above identity for any operator satisfying \eqref{eq:BaxterFormula} in a given algebra, and then applied this result for the above mentioned Rota-Baxter algebra of functions $\varphi$. During the next forty years the result of Baxter has been discussed in a narrow range and the greatest contribution has been made by Gian Carlo Rota \cite{rota1969baxter,rota1969baxterII,rota1972fluctuation,rota1995baxter}, Pierre Cartier \cite{cartier1972structure}, John F.C. Kingman \cite{kingman1962spitzer}, Frederic V. Atkinson \cite{atkinson1963some}, 
and a few others. The big breakthrough has been started 
in 2000 by Li Guo, and later on by his collaborators, who investigate Rota-Baxter algebras in many different directions. From the point of view of this article the most important achievement is a generalization of Spitzer's identity to non-commutative Rota-Baxter algebras given by Kurusch Ebrahimi-Fard, Li Guo and Dirk Kreimer in the context of renormalization in perturbative quantum field theory \cite{ebrahimi2004spitzer} (see also \cite{ebrahimi2006birkhoff, guo2012introduction}).

In this article we derive a certain generalization of Spitzer's identity. Namely, for a commutative Rota-Baxter algebra $(\Algebra,\RotaOp)$ of weight $\lambda$ we show that
\begin{align}
\label{eq:TheEquality}
\sum_{n=0}^\infty \underbrace{\RotaOp(a_1\cdots\RotaOp(a_1}_{n}\RotaOp((1 + \lambda a_1)a_0))\cdots)  
= \Exp {\RotaOp\left(\lambda^{-1}\log(1+\lambda a_1) \right) } 
 \RotaOp\left(\Exp {-\RotaOp\left(\lambda^{-1}\log(1+\lambda a_1) \right) }a_0\right)
\end{align}
 if $\lambda\in\Ring$ is not a zero divisor in $\Algebra$, and
\begin{align}
\label{eq:TheEqualityZero}
\sum_{n=0}^\infty \underbrace{\RotaOp(a_1\cdots\RotaOp(a_1}_{n}\RotaOp(a_0)\cdots) 
= \Exp{\RotaOp(a_1)}\RotaOp\Bra{\Exp{-\RotaOp(a_1)}a_0}
\end{align}
if $\lambda = 0$. Also for a non-commutative Rota-Baxter algebra $(\Algebra,\RotaOp)$ of weight $\lambda$ (assuming $\lambda\in\Ring$ is not a zero divisor in $\Algebra$) we show that
\begin{align}
\label{eq:TheEqualityNonCom}
\sum_{n=0}^\infty \underbrace{\RotaOp(a_1\cdots\RotaOp(a_1}_{n}\RotaOp((1 + \lambda a_1)a_0))\cdots) 
= \Exp {\left(\RotaOp\left(\CHL\Bra{\lambda^{-1}\log(1+\lambda a_1)} \right) \right)} 
 \RotaOp\left(\Exp {\left(-\RotaOp\left(\CHL\Bra{\lambda^{-1}\log(1+\lambda a_1)} \right) \right)}a_0\right)
\end{align}
where  $\CHL$ is the $\BCH$-recursion introduced in \cite{ebrahimi2004spitzer} and defined by \eqref{eq:CHL} in section \ref{sec:Result}. Similarly, for a non-commutative Rota-Baxter algebra $(\Algebra,\RotaOp)$ of weight $0$ it occurs that
\begin{align}
\label{eq:TheEqualityNonComZero}
\sum_{n=0}^\infty \underbrace{\RotaOp(a_1\cdots\RotaOp(a_1}_{n}\RotaOp(a_0))\cdots) 
= \Exp{\RotaOpBra{\CHZ(a_1)}}
 \RotaOp\left(\Exp{-\RotaOpBra{\CHZ(a_1)}}a_0\right).
\end{align}
where  $\CHZ$ is the zero $\BCH$-recursion introduced in \cite{ebrahimi2006birkhoff} and defined by \eqref{eq:CHZ} in section \ref{sec:Result}.

These results are collected in Theorems \ref{thm:MainCommutative} and \ref{thm:MainNonCommutative} and Corollary \ref{cor:MainNonCommutativeZero} in section \ref{sec:Result}. They are preceded by the definition of complete filtered Rota-Baxter algebra, certain facts used in the article, and the origin of these results. In section \ref{sec:Result} we also post Propositions \ref{prop:Eulerian} and \ref{prop:EulerianTwo} containing certain Eulerian identities. In section \ref{sec:Proof} we prove the two stated theorems, and in section \ref{sec:Eulerian} we apply Theorem \ref{thm:MainCommutative} in 
the algebra $\QQ[[t]]$, with  $q$-integral as the Rota-Baxter operator, to show the Eulerian identities stated in the propositions.  
 
\section{Results}
\label{sec:Result}

We begin by introducing the most important notions and facts about Rota-Baxter algebras (we will follow \cite{guo2012introduction}).
Let $\Ring$ be a commutative ring and $\Algebra$ be a $\Ring$-algebra. If there exist $\lambda\in\Ring$ and a linear operator $\RotaOp : \Algebra \to \Algebra$ satisfying the Rota-Baxter equation
\begin{align}
 \label{eq:RotaBaxter}
 \RotaOp(x)\RotaOp(y) = \RotaOp(x\RotaOp(y)) + \RotaOp(\RotaOp(x)y) + \lambda \RotaOp(xy)
\end{align}
for all $x,y\in\Algebra$, then $(\Algebra,\RotaOp)$ is called a Rota-Baxter algebra of weight $\lambda$, and $\RotaOp$ is called a Rota-Baxter operator of weight $\lambda$.  If $\Algebra$ is additionally commutative, we say $(\Algebra,\RotaOp)$ is a commutative Rota-Baxter algebra of weight $\lambda$.
Set $\RotaOpT: \Algebra \to \Algebra$, 
$$
\RotaOpT(x) = -\lambda x - \RotaOp(x).
$$ 
Then $(\Algebra,\RotaOpT)$ is also a Rota-Baxter algebra of weight $\lambda$, as one can easily check. 

In order to consider the exponential and the logarithmic functions in the algebra we need to assure convergence of series. Therefore we assume $\Algebra$ is a filtered algebra, namely for $n\in\NN\cup\Set{0}$ there exists a non-unitary subalgebra $\Algebra_n \subset \Algebra$ such that  $\Algebra = \Algebra_0$, $\Algebra_{n+1}\subset \Algebra_n$, $\bigcap_n \Algebra_n = \Set{0}$, and for all $n,m \in\NN\cup\Set{0}$ we have $\Algebra_n\Algebra_m \subset \Algebra_{n+m}$. In the filtered algebra we define a metric $d:\Algebra\times\Algebra\to \RR$ given by 
\begin{align*}
 d(x,y) = \inf\SetSuchThat{2^{-n}}{x-y \in \Algebra_n}.
\end{align*}
We say that an algebra $\Algebra$ with filtration $\SetSuchThat{\Algebra_n}{n\in\NN\cup\Set{0}}$ is a complete filtered algebra if $d$ is a complete metric on $\Algebra$. 
We say that $(\Algebra,\Algebra_n,\RotaOp)$ is a complete filtered Rota-Baxter algebra of weight $\lambda$ if $\Algebra$ is a complete filtered algebra and $\RotaOp(\Algebra_n)\subset\Algebra_n$ for all $n\in\NN\cup\Set{0}$.
It is easy to see that for a sequence $x_n\in\Algebra_n$, $n\in\NN\cup\Set{0}$, a series $\sum_{n=0}^\infty x_n$ is convergent in $\Algebra$. Therefore, the exponential $\exp:\Algebra_1\to\Algebra$ and the logarithmic $\log:1 + \Algebra_1\to\Algebra$ functions given by the standard formulas
\begin{align*}
 \Exp x &= \exp x = \sum_{n=0}^\infty \frac {x^n}{n!} , & \log (1 + x) &= - \sum_{n=1}^\infty \frac {(-1)^n}{n} x^n
\end{align*}
are well defined.
We will use a well known facts that $\exp\log(1+x) = 1+x$ and $\log\exp(x) = x$ for all $x\in\Algebra_1$.

We are now ready to present results of this article. The most significant identity satisfied in a (commutative) Rota-Baxter algebra is Spitzer's identity introduced first by Spitzer \cite{spitzer1956combinatorial} in the probability theory, and then abstract-algebraically described by Baxter \cite{baxter1960analytic} and Rota \cite{rota1969baxter}. In a modern language Spitzer's identity can be stated as follows \cite{guo2012introduction}. Let $(\Algebra,\Algebra_n,\RotaOp)$ be a commutative complete Rota-Baxter algebra of weight $\lambda$ such that $\lambda$ is not a zero divisor of $\Algebra$, and $a\in\Algebra_1$ is fixed. Then the equation
\begin{align}
 \label{eq:LinearHomogenEq}
 b = 1 + \RotaOp(ab)
\end{align}
has a unique solution
\begin{align*}
 b = \exp \left(\RotaOp\left(\lambda^{-1}\log(1+\lambda a) \right) \right),
\end{align*}
where, here and throughout the article, we use an abbreviation
\begin{align*}
 \lambda^{-1}\log(1+\lambda a) =  \sum_{n=1}^\infty \frac {(-\lambda)^{n-1}}{n} a^n,
\end{align*}
(so we do not need to assume $\lambda$ to be invertible in $\Ring$). 

Iterating the equation \eqref{eq:LinearHomogenEq}, it is easy to see that
\begin{align*}
 b = \sum_{n=0}^\infty \underbrace{\RotaOp(a\cdots\RotaOp(a}_{n})\cdots).
\end{align*}
Spitzer's identity is thus
\begin{align}
\label{eq:Spitzer}
 \exp \left(\RotaOp\left(\lambda^{-1}\log(1+\lambda a) \right) \right) =\sum_{n=0}^\infty \underbrace{\RotaOp(a\cdots\RotaOp(a}_{n})\cdots).
\end{align}
If $\lambda = 0$, then Spitzer's identity is even simpler
\begin{align*}
 \exp \left(\RotaOp(a) \right) = \sum_{n=0}^\infty \underbrace{\RotaOp(a\cdots\RotaOp(a}_{n})\cdots).
\end{align*}

Let us look at a certain example. Let $\Functions$ be the $\RR$-algebra of analytic functions on the real line, $\Functions_n\subset \Functions$ the functions $f:\RR\to\RR$ such that  $f(0) = f'(0) = \ldots = f^{(n-1)}(0) = 0$ (for $n\in\NN$), and $\RotaInt : \Functions \to \Functions$ be the integral operator given by
\begin{align}
\label{eq:RotaInt}
 \RotaInt(f)(t) = \int_0^t f(s) ds.
\end{align}
Using integration by parts formula, it easy to prove that $(\Functions, \Functions_n, \RotaInt)$ is a commutative complete filtered Rota-Baxter algebra of weight $0$. The equation \eqref{eq:LinearHomogenEq} in this algebra reads as 
\begin{align*}
 b(t) = 1 + \int_0^t a_1(s)b(s) ds.
\end{align*}
By the Picard iteration of this integral equation we get
\begin{align*}
 b(t) = \sum_{n=0}^\infty \underbrace{\RotaInt(a\cdots\RotaInt(a}_{n})\cdots)(t),
\end{align*}
On the other hand, differentiating it, we obtain a non-homogeneous non-autonomous linear differential equation
\begin{align*}
 \dot b(t) =  a_1(t)b(t), \qquad b(0) = 1,
\end{align*}
with a well known solution
\begin{align*}
 b(t) = \Exp{\RotaInt(a_1)(t)}
\end{align*}
Observe that  Spitzer's identity in $(\Functions, \Functions_n, \RotaInt)$ is nothing else but comparing these two formulas for $b(t)$.
Now, in a commutative complete filtered Rota-Baxter algebra $(\Algebra,\Algebra_n,\RotaOp)$ of weight $0$, consider the following generalization of the equation \eqref{eq:LinearHomogenEq} 
\begin{align}
 \label{eq:One}
 b = \RotaOp(a_0) + \RotaOp(a_1b),
\end{align}
 where $a_0,a_1\in\Algebra_1$ are fixed. In the special case $(\Functions, \Functions_n, \RotaInt)$, this equation is 
\begin{align*}
 b(t) = \int_0^t a_0(s) ds + \int_0^t a_1(s)b(s) ds.
\end{align*}
Once again, by the Picard iteration we obtain
\begin{align*}
 b(t) = \sum_{n=0}^\infty \underbrace{\RotaInt(a_1\cdots\RotaInt(a_1}_{n}\RotaInt(a_0))\cdots)(t),
\end{align*}
and on the other hand, differentiating it, we obtain a non-homogeneous non-autonomous linear differential equation
\begin{align*}
 \dot b(t) = a_0(t) + a_1(t)b(t), \qquad b(0) = 0,
\end{align*}
with a well known solution
\begin{align*}
 b(t) = \Exp{\RotaInt(a_1)(t)}\RotaInt(\Exp{-\RotaInt(a_1)}a_0)(t).
\end{align*}
This suggests that the solution of \eqref{eq:One} for any Rota-Baxter algebra of weight $0$ is
\begin{align*}
 b = \sum_{n=0}^\infty \underbrace{\RotaOp(a_1\cdots\RotaOp(a_1}_{n}\RotaOp(a_0))\cdots)(t) 
 = \exp(\RotaOp(a_1))\RotaOp(\exp(-\RotaOp(a_1))a_0).
\end{align*}
A natural question arise: if this formula can be extended for $\lambda \neq 0$. The answer is positive if we modify $a_0$ by a factor $1 + \lambda a_1$. We state this result in the following theorem.

\begin{theorem}
\label{thm:MainCommutative}
 Let $(\Algebra,\Algebra_n,\RotaOp)$ be a commutative complete filtered Rota-Baxter algebra of weight $\lambda$, and $a_0,a_1 \in\Algebra_1$ are fixed. Then the equation
\begin{align}
 \label{eq:LinearInhomogenEq}
b = \RotaOp((1 + \lambda a_1)a_0) + \RotaOp(a_1b)
\end{align}
\begin{enumerate}[(i)]
 \item has a unique solution
\begin{align}
 \label{eq:LinearInhomogenEqSolution}
 b = \exp \left(\RotaOp\left(\lambda^{-1}\log(1+\lambda a_1) \right) \right) 
 \RotaOp\left(\exp \left(-\RotaOp\left(\lambda^{-1}\log(1+\lambda a_1) \right) \right)a_0\right)
\end{align}
in case $\lambda$ is not a zero divisor of $\Algebra$, and moreover the equality \eqref{eq:TheEquality} is satisfied.
 \item has a unique solution
\begin{align}
 \label{eq:LinearInhomogenEqSolutionZero}
b = \exp(\RotaOp(a_1))\RotaOp(\exp(-\RotaOp(a_1))a_0)
\end{align}
in case $\lambda = 0$, and moreover the equality \eqref{eq:TheEqualityZero} is satisfied.
\end{enumerate}

\end{theorem}

As an application of Theorem \ref{thm:MainCommutative} we show an Eulerian identity.
\begin{proposition}
\label{prop:Eulerian}
 For $1\neq q\in\QQ$ the following equality holds true
 \begin{align*}
 1 + \sum_{n=1}^\infty \frac{q^{2n-1}t^{n}}{(1 -q)\cdots(1-q^{n})} = (1-t)\prod_{n=1}^\infty \frac 1 {1-q^nt}.
\end{align*}
\end{proposition}

For a non-commutative algebra the theorem must be modified. Following \cite{ebrahimi2004spitzer, ebrahimi2005integrable, ebrahimi2006birkhoff,guo2012introduction} in a non-commutative complete filtered Rota-Baxter algebra of weight $\lambda$ ($\lambda$ is not a zero divisor of $\Algebra$) we introduce the $\BCH$-recursion operator $\CHL : \Algebra_1 \to \Algebra_1$ which is defined as the unique solution of the algebraic equation
\begin{align}
\label{eq:CHL}
 \CHL(a) = a + \lambda^{-1}\BCH\Bra{\RotaOp\Bra{\CHL(a)},\RotaOpT\Bra{\CHL(a)}}
\end{align}
for $a\in\Algebra_1$. Here $\BCH : \Algebra_1 \times \Algebra_1 \to \Algebra_1$ is the celebrated Baker-Campbell-Hausdorff power series in a non-commutative algebra given as the unique solution of 
\begin{align*}
 \exp(x)\exp(y) = \exp\Bra{x+ y + \BCH(x,y)},
\end{align*}
for all $x,y\in\Algebra_1$. The operator $\CHL$ is introduced so that the formula
\begin{align}
 \label{eq:BCH-CHL}
 \exp(-\lambda a) = \exp (\RotaOp\Bra{\CHL(a)})\exp\Bra{\RotaOpT\Bra{\CHL(a)}}
\end{align}
be fulfilled for all $a\in\Algebra_1$. Having $\CHL$ we can state the analogue of Theorem \ref{thm:MainCommutative} in the non-commutative algebra.

\begin{theorem}
\label{thm:MainNonCommutative}
 Let $(\Algebra,\Algebra_n,\RotaOp)$ be a (non-commutative) complete filtered Rota-Baxter algebra of weight $\lambda$ which is not a zero divisor of $\Algebra$, and $a_0,a_1 \in\Algebra_1$ are fixed. Then 
 \begin{enumerate}[(i)]
  \item  the equation
\begin{align}
 \label{eq:LinearInhomogenEqNonCom}
b = \RotaOp((1 + \lambda a_1)a_0) + \RotaOp(a_1b)
\end{align}
has the unique solution
\begin{align}
 \label{eq:LinearInhomogenEqSolutionNonCom}
 b = \exp \left(\RotaOp\left(\CHL\Bra{\lambda^{-1}\log(1+\lambda a_1)} \right) \right) 
 \RotaOp\left(\exp \left(-\RotaOp\left(\CHL\Bra{\lambda^{-1}\log(1+\lambda a_1)} \right) \right)a_0\right),
\end{align}
and moreover the equality \eqref{eq:TheEqualityNonCom} is satisfied;
 \item the equation
\begin{align*}
b = \RotaOpT(a_0(1 + \lambda a_1)) - \RotaOpT(ba_1)
\end{align*}
has the unique solution
\begin{align*}
 b = 
 \RotaOp\left(a_0\exp \left(-\RotaOp\left(\CHL\Bra{\lambda^{-1}\log(1+\lambda a_1)} \right) \right)\right)
 \exp \left(\RotaOp\left(\CHL\Bra{\lambda^{-1}\log(1+\lambda a_1)} \right) \right) .
\end{align*}
\end{enumerate}

\end{theorem}

If $\lambda = 0$ the theorem must be additionally modified. As was pointed out in \cite{ebrahimi2006birkhoff, carinena2007hopf, ebrahimi2009magnus} the $\BCH$-recursion $\CHL$ reduces for $\lambda \to 0$ to the Magnus recursion \cite{magnus1954exponential}. Namely, for a non-commutative complete filtered Rota-Baxter algebra of weight $0$ the zero $\BCH$-recursion operator $\CHZ : \Algebra_1 \to \Algebra_1$  is given by the unique solution of the algebraic equation
\begin{align}
\label{eq:CHZ}
 \CHZ(a) &=  \frac{\ad{\RotaOpBra{\CHZ(a)}}} {\exp(\ad{\RotaOpBra{\CHZ(a)}}) - 1} (a) \\
\nonumber &= \Bra{ 1 + \sum_{k=1}^\infty \frac {B_k} {k!} \Bra{\ad{\RotaOpBra{\CHZ(a)}}}^k } (a) ,
\end{align}
for each $a\in\Algebra_1$. Here, $\ad{x}(y) = [x,y] = xy - yx$ is the linear addjoint operator, and $B_k$, $k\in\NN$, are the 
Bernoulli numbers. Then the solution of the equation \eqref{eq:LinearHomogenEq} is $b = \exp(\RotaOpBra{\CHZ(a)})$ \cite{ebrahimi2006birkhoff}, so that 
$\exp\left(\RotaOp\left(\CHL\Bra{\lambda^{-1}\log(1+\lambda a_1)}\right) \right)$ reduces to $\exp(\RotaOpBra{\CHZ(a)})$ in the 
limit $\lambda \to 0$. Therefore we can state the following corollary from Theorem \ref{thm:MainNonCommutative}.

\begin{corollary}
\label{cor:MainNonCommutativeZero}
 Let $(\Algebra,\Algebra_n,\RotaOp)$ be a (non-commutative) complete filtered Rota-Baxter algebra of weight $0$, and $a_0,a_1 \in\Algebra_1$ are fixed. Then   the equation
\begin{align*}
b = \RotaOp(a_0) + \RotaOp(a_1b)
\end{align*}
has the unique solution
\begin{align}
\label{eq:LinearInhomogenEqSolutionNonComZero}
 b =  \exp(\RotaOpBra{\CHZ(a_1)})
 \RotaOp\left(\exp(-\RotaOpBra{\CHZ(a_1)})a_0\right),
\end{align}
and moreover the equality \eqref{eq:TheEqualityNonComZero} is satisfied.
\end{corollary}

In the context of the algebra of linear operators, considered by Wilhelm Magnus in \cite{magnus1954exponential}, this corollary is quite clear. Indeed, knowing (by Magnus) that the solution of a homogeneous linear equation
\begin{align*}
 \dot Y(t) = A_1(t)Y(t), \qquad Y(0) = \Id,
\end{align*}
is $Y(t) = \exp(\Omega(t))$, where $\Omega(t)$ satisfies $\dot \Omega(t) =  \frac{\ad{{\Omega(t)}}} {\exp(\ad{{\Omega(t)}}) - 1} (A_1)$, $\Omega(0) = 0$, it is known that for an inhomogeneous equation
\begin{align}
\label{eq:LinearInhomogenDiff}
 \dot B(t) = A_0(t) + A_1(t)B(t), \qquad B(0) = 0,
\end{align}
the solution is 
\begin{align}
\label{eq:LinearInhomogenDiffSolution}
B(t) = \exp(\Omega(t))\int_0^t \exp(-\Omega(s))A_0(s)\, ds.
\end{align}
Now $\RotaInt$ defined as in \eqref{eq:RotaInt} is also a Rota-Baxter operator in the space of linear operators, and one can conclude, as observed in \cite{ebrahimi2006birkhoff}, that $\Omega(t) = \RotaInt(\CHZ( A_0))$. Therefore, we see that \eqref{eq:LinearInhomogenDiffSolution} is equivalent to \eqref{eq:LinearInhomogenEqSolutionNonComZero}. Integrating \eqref{eq:LinearInhomogenDiff} (actually using $\RotaInt$) we have Corollary \ref{cor:MainNonCommutativeZero} for $\RotaOp = \RotaInt$.

Before we end this section let us look at a certain interesting case of the equation \eqref{eq:LinearInhomogenEq}. Namely, assume $(1 + \lambda a_1)a_0 = - a_1$, so that $a_0 = -(1 + \lambda a_1)^{-1}a_1$, where we abbreviate
\begin{align*}
 (1 + \lambda a_1)^{-1} = \sum_{n=0}^\infty (-\lambda a_1)^n.
\end{align*}
In this case formula \eqref{eq:TheEquality} reads as
\begin{align*}
-\Exp {\RotaOp\left(\lambda^{-1}\log(1+\lambda a_1) \right) } 
 \RotaOp\left(\Exp {-\RotaOp\left(\lambda^{-1}\log(1+\lambda a_1) \right) }(1 + \lambda a_1)^{-1}a_1\right)
   = 1 - \sum_{n=0}^\infty \underbrace{\RotaOp(a_1\cdots\RotaOp(a_1}_{n})\cdots).
\end{align*}
By  Spitzer's identity the sum on the right side 
is equal to  
$\exp (\RotaOp\left(\lambda^{-1}\log(1+\lambda a_1) \right) )$. Using this fact we can transform this equality into the following form
\begin{align}
\label{eq:SpecialEquality}
 {1 - \RotaOp\left(\Exp {-\RotaOp\left(\lambda^{-1}\log(1+\lambda a_1) \right) }(1 + \lambda a_1)^{-1}a_1\right)} = \Exp {-\RotaOp\left(\lambda^{-1}\log(1+\lambda a_1) \right) }.
\end{align}
This implies that $\Exp {-\RotaOp\left(\lambda^{-1}\log(1+\lambda a_1) \right) }$ is a solution of an equation 
\begin{align}
\label{eq:EquationD}
d = 1 + \RotaOp(-(1 + \lambda a_1)^{-1}a_1 d).
\end{align}
But from  Spitzer's identity this equation has also a solution 
$\Exp {\RotaOp\left(\lambda^{-1}\log(1 -(1 + \lambda a_1)^{-1}a_1)\right)}$. Both formulas coincides since 
\begin{align*}
 -\log(1+\lambda a_1) = \log\Bra{\frac 1 {1+\lambda a_1}} = \log\Bra{1 -\frac {\lambda a_1}{1 + \lambda a_1}}.
\end{align*}
This reasoning explains roughly the occurrence of the factor $1+\lambda a_1$ in the first component of the right side of \eqref{eq:LinearInhomogenEq}. It is needed because the ``inverse dynamics'' to the ``dynamics'' generated by the homogeneous equation
\begin{align}
\label{eq:EquationC}
 c = 1 + \RotaOp(a_1 c)
\end{align}
is given by the ``dynamics'' generated by the equation \eqref{eq:EquationD}. Observe also that with the above remarks we can say that the solution of the equation \eqref{eq:LinearInhomogenEq} is $b = c\RotaOp(d a_0)$, where $c$ and $d$ are the unique solutions of the equations \eqref{eq:EquationC} and \eqref{eq:EquationD}, respectively.

It occurs that formula \eqref{eq:SpecialEquality} gives another Eulerian identity.
\begin{proposition}
\label{prop:EulerianTwo}
 For $1\neq q\in\QQ$ the following equality holds true
 \begin{align*}
 1 + \sum_{n=1}^\infty \frac{q^{n}t^{n}}{(1 -q)\cdots(1-q^{n})} = \prod_{n=1}^\infty \frac 1 {1-q^nt}.
\end{align*}
\end{proposition}

\section{Proofs of Theorems}
\label{sec:Proof}

In this section we prove Theorem \ref{thm:MainCommutative} and Theorem \ref{thm:MainNonCommutative}. In what follows we use
formula (see \cite{kingman1962spitzer, guo2012introduction})
\begin{align}
 \label{eq:Kingman}
 \lambda\RotaOp(u)^n = \RotaOpBra{\left(-\RotaOpTBra{u}\right)^n - \Bra{\RotaOpBra{u}}^n }
\end{align}
fulfilled for every $u\in\Algebra$, where $(\Algebra,\RotaOp)$ is a Rota-Baxter algebra of weight $\lambda$ with the assumption that $\lambda$ is not a zero divisor of $\Algebra$.
On the other hand for $\lambda = 0$ we need the following lemma.

\begin{lemma}
\label{lem:Iteration}
Let $(\Algebra, \RotaOp)$ be a commutative Rota-Baxter algebra of weight $0$. Then for every $a \in \Algebra$ and $k\in\NN \cup \Set{0}$ we have
\begin{enumerate}[(i)]
 \item \label{item:A} $\underbrace{\RotaOp(a\cdots\RotaOp(a}_{k})\cdots) = \frac 1 {k!} \Bra{\RotaOp(a)}^k$; 
 \item \label{item:B} $\sum_{l=0}^k (-1)^l \underbrace{\RotaOp(a\cdots\RotaOp(a}_{k+1-l})\cdots)
 \underbrace{\RotaOp(a\cdots\RotaOp(a}_{l})\cdots) =
 (-1)^k\underbrace{\RotaOp(a\cdots\RotaOp(a}_{k+1})\cdots)$.
\end{enumerate}
\end{lemma}

\begin{proof}
 We prove \eqref{item:A} by induction on $k$. For $k=0$ the equality is obvious. Assume $$\underbrace{\RotaOp(a\cdots\RotaOp(a}_{n})\cdots) = \frac 1 {n!} \Bra{\RotaOp(a)}^n$$ 
 is satisfied for all $n \leq k$. Using the induction hypothesis and then Rota-Baxter formula \eqref{eq:RotaBaxter}  for $x = a (\RotaOp(a))^{k-1}$ and $y=a$ we get
 \begin{align*}
  \underbrace{\RotaOp(a\cdots\RotaOp(a}_{k+1})\cdots) 
  & = \frac 1 {k!} \RotaOp\Bra{a (\RotaOp(a))^k} \\
  & = \frac 1 {k!} \Bra{ \RotaOp\Bra{a (\RotaOp(a))^{k-1}}\RotaOp(a) - \RotaOp\Bra{\RotaOp\Bra{a (\RotaOp(a))^{k-1}} a}}.
 \end{align*}
Using the induction hypothesis twice for both components we obtain
\begin{align*}
   & = \frac 1 {k!} \Bra{ (k-1)!\underbrace{\RotaOp(a\cdots\RotaOp(a}_{k})\cdots)\RotaOp(a) - (k-1)!\underbrace{\RotaOp(a\cdots\RotaOp(a}_{k+1})\cdots)} \\
  & = \frac 1 {k} \Bra{ \frac 1 {k!} \Bra{\RotaOp(a)}^{k+1} - \frac 1 {(k+1)!} \Bra{\RotaOp(a)}^{k+1}} \\
  & =  \frac 1 {(k+1)!} \Bra{\RotaOp(a)}^{k+1}.
\end{align*}
This ends the proof of \eqref{item:A}.

For \eqref{item:B} it is enough to use result from \eqref{item:A} and a well known formula
\begin{align*}
 \sum_{l=0}^{k+1} \frac {(-1)^l} {(k+1 - l)!\, l!} = 0. 
\end{align*}
\end{proof}

\begin{proof}[Proof of Theorem \ref{thm:MainCommutative}.]
First, let us prove the uniqueness of a solution in both cases. Suppose $b, b' \in \Algebra$ satisfy \eqref{eq:LinearInhomogenEq}. Then $b - b' = \RotaOp(a_1(b - b'))$. Since $a_1 \in \Algebra_1$, by induction on $n\in\NN\cup\Set{0}$, we get that $b-b' \in \Algebra_n$ for all $n\in\NN$. By completeness of $\Algebra$, $b-b' \in \bigcap_n \Algebra_n = \Set 0 $. So $b=b'$.

Now we prove (i), i.e., that \eqref{eq:LinearInhomogenEqSolution} is a solution of \eqref{eq:LinearInhomogenEq} in case $\lambda$ is not a zero divisor of $\Algebra$.
Let $u = \lambda^{-1}\log(1+\lambda a_1)$, so that $a_1 = \lambda^{-1}(\exp(\lambda u)-1)$. In order to show that \eqref{eq:LinearInhomogenEqSolution} is a solution of \eqref{eq:LinearInhomogenEq} it is enough to show
\begin{align}
\label{eq:ToProve}
 \lambda \Exp{\RotaOp(u)}\RotaOp\left(\Exp{-\RotaOp(u)}a_0\right) = 
 \lambda\RotaOp(\Exp{\lambda u}a_0) + \RotaOp\left( \left(\Exp{\lambda u}-1\right) \Exp{\RotaOp(u)}\RotaOp\left(\Exp{-\RotaOp(u)}a_0\right) 
  \right).
\end{align}
Observe that 
\begin{align*}
  \left(\Exp{\lambda u}-1\right) \Exp{\RotaOp(u)}\RotaOp\left(\Exp{-\RotaOp(u)}a_0\right) & =
  \left(\Exp{-\RotaOpT(u)} - \Exp{\RotaOp(u)}\right) \RotaOp\left(\Exp{-\RotaOp(u)}a_0\right).
\end{align*}
Then, using \eqref{eq:RotaBaxter} for $x = \Exp{-\RotaOpT(u)} - \Exp{\RotaOp(u)}$ and $y = \Exp{-\RotaOp(u)}a_0$ we get
\begin{align*}
  \RotaOp\left(\left(\Exp{\lambda u}-1\right) \Exp{\RotaOp(u)}\RotaOp\left(\Exp{-\RotaOp(u)}a_0\right) \right) & =
  \RotaOp\left( \Exp{-\RotaOpT(u)} - \Exp{\RotaOp(u)} \right) \RotaOp\left(  \Exp{-\RotaOp(u)}a_0  \right) 
  \\ & \quad
  - \RotaOp\left( \RotaOp\left(\Exp{-\RotaOpT(u)} - \Exp{\RotaOp(u)} \right) \Exp{-\RotaOp(u)}a_0   \right)
  \\ & \quad
  - \RotaOp\left( \lambda\left(\Exp{-\RotaOpT(u)} - \Exp{\RotaOp(u)}\right)    \Exp{-\RotaOp(u)}a_0  \right).
\end{align*}
By the definition of $\exp$ and formula \eqref{eq:Kingman} we compute that
\begin{align*}
 \RotaOp\left( \Exp{-\RotaOpT(u)} - \Exp{\RotaOp(u)} \right)  & = \lambda \Bra{\Exp{\RotaOpBra{u}} - 1}. 
\end{align*}
Therefore
\begin{multline}
\label{eq:Almost}
  \RotaOp\left( \left(\Exp{\lambda u}-1\right) \Exp{\RotaOp(u)}\RotaOp\left(\Exp{-\RotaOp(u)}a_0\right) \right)  
  =  \\ 
  = \lambda \Bra{\Exp{\RotaOpBra{u}} - 1} \RotaOpBra{\Exp{-\RotaOpBra u}a_0}
  - \lambda \RotaOpBra{\Bra{\Exp{\RotaOpBra{u}} - 1} {\Exp{-\RotaOpBra u}a_0}} \\
  - \RotaOp\left( \lambda\left(\Exp{-\RotaOpT(u)} - \Exp{\RotaOp(u)}\right)    \Exp{-\RotaOp(u)}a_0  \right)  \\
  = \lambda \Bra{\Exp{\RotaOpBra{u}} - 1} \RotaOpBra{\Exp{-\RotaOpBra u}a_0} 
  - \lambda\RotaOp\left(\left(\Exp{-\RotaOpT(u)} - 1\right)    \Exp{-\RotaOp(u)}a_0  \right).
\end{multline}
Since $-\RotaOpTBra{u} - \RotaOpBra{u} = \lambda u$, it is easy to see that the right side transforms to
\begin{align*}
 \RotaOp\left( \left(\Exp{\lambda u}-1\right) \Exp{\RotaOp(u)}\RotaOp\left(\Exp{-\RotaOp(u)}a_0\right) \right) =
 \lambda\Bra{\Exp{\RotaOpBra u}\RotaOpBra{\Exp{-\RotaOpBra u}a_0} - \RotaOpBra{\Exp{\lambda u}a_0}}.
\end{align*}
This is exactly \eqref{eq:ToProve}.
Finally, it is easy to see that
\begin{align*}
 \sum_{n=0}^\infty \underbrace{\RotaOp(a_1\cdots\RotaOp(a_1}_{n}\RotaOp((1 + \lambda a_1)a_0))\cdots) 
\end{align*}
is a solution of \eqref{eq:LinearInhomogenEq}. Since $n$-th component of this sum is in $\Algebra_{n+1}$, the series is convergent. From the uniqueness of the solution we conclude \eqref{eq:TheEquality}.

In order to prove (ii), i.e., that \eqref{eq:LinearInhomogenEqSolutionZero} is a solution of \eqref{eq:LinearInhomogenEq} for $\lambda = 0$ we proceed slightly differently. By Lemma \ref{lem:Iteration} \eqref{item:A} and the definition of $\exp$ the formula \eqref{eq:LinearInhomogenEqSolutionZero} is equivalent to
\begin{align*}
 b & = \sum_{n=0}^\infty \sum_{m=0}^\infty \underbrace{\RotaOp(a_1\cdots\RotaOp(a_1}_{n})\cdots)
 \RotaOp\Bra{(-1)^m\underbrace{\RotaOp(a_1\cdots\RotaOp(a_1}_{m})\cdots) a_0} \\
 & = \sum_{k=0}^\infty \sum_{l=0}^k (-1)^l \underbrace{\RotaOp(a_1\cdots\RotaOp(a_1}_{k-l})\cdots)
 \RotaOp\Bra{\underbrace{\RotaOp(a_1\cdots\RotaOp(a_1}_{l})\cdots) a_0}.
\end{align*}
On the other hand, by iterating the considered equation \eqref{eq:LinearInhomogenEq} it is easy to see that
\begin{align*}
 \sum_{k=0}^\infty \underbrace{\RotaOp(a_1\cdots\RotaOp(a_1}_{k}\RotaOp(a_0))\cdots) 
\end{align*}
is also a solution of \eqref{eq:LinearInhomogenEq}. To complete the proof of (ii) we show, using induction on $k$, that
\begin{align*}
 \underbrace{\RotaOp(a_1\cdots\RotaOp(a_1}_{k}\RotaOp(a_0))\cdots) =
 \sum_{l=0}^k (-1)^l \underbrace{\RotaOp(a_1\cdots\RotaOp(a_1}_{k-l})\cdots)
 \RotaOp\Bra{\underbrace{\RotaOp(a_1\cdots\RotaOp(a_1}_{l})\cdots) a_0}.
\end{align*}
For $k=0$ on both sides there is $\RotaOp(a_0)$. Assuming the above equality is correct, we show
\begin{align*}
 \RotaOp\Bra{a_1\underbrace{\RotaOp(a_1\cdots\RotaOp(a_1}_{k}\RotaOp(a_0))\cdots)} = 
\sum_{l=0}^k (-1)^l \RotaOp\Bra{a_1\underbrace{\RotaOp(a_1\cdots\RotaOp(a_1}_{k-l})\cdots)
 \RotaOp\Bra{\underbrace{\RotaOp(a_1\cdots\RotaOp(a_1}_{l})\cdots) a_0}}
 \end{align*}
From the R-B formula \eqref{eq:RotaBaxter}  for $x = a_1\underbrace{\RotaOp(a_1\cdots\RotaOp(a_1}_{k-l})\cdots)$ 
and $y=\underbrace{\RotaOp(a_1\cdots\RotaOp(a_1}_{l})\cdots) a_0$ the right side is equal
\begin{align*}
 & = \sum_{l=0}^k (-1)^l \underbrace{\RotaOp(a_1\cdots\RotaOp(a_1}_{k+1-l})\cdots)
 \RotaOp\Bra{\underbrace{\RotaOp(a_1\cdots\RotaOp(a_1}_{l})\cdots) a_0} \\
 & \qquad - \sum_{l=0}^k (-1)^l \RotaOp\Bra{\underbrace{\RotaOp(a_1\cdots\RotaOp(a_1}_{k+1-l})\cdots)
 \underbrace{\RotaOp(a_1\cdots\RotaOp(a_1}_{l})\cdots) a_0}.
\end{align*}
Using Lemma \ref{lem:Iteration}\eqref{item:B} for the second line we conclude that
\begin{align*}
 \underbrace{\RotaOp(a_1\cdots\RotaOp(a_1}_{k+1}\RotaOp(a_0))\cdots) =
 \sum_{l=0}^{k+1} (-1)^l \underbrace{\RotaOp(a_1\cdots\RotaOp(a_1}_{k+1-l})\cdots)
 \RotaOp\Bra{\underbrace{\RotaOp(a_1\cdots\RotaOp(a_1}_{l})\cdots) a_0},
\end{align*}
which is what we want to prove.

Finally, the generalized Spitzer's identities \eqref{eq:TheEquality} and \eqref{eq:TheEqualityZero} holds true by repeated use of \eqref{eq:LinearInhomogenEq}.
\end{proof}

\begin{proof}[Proof of Theorem \ref{thm:MainNonCommutative}.]
The uniqueness of the solution in both cases comes by the same reasoning as in the proof of Theorem \ref{thm:MainCommutative}.

The proof of $(i)$ is also similar to that of Theorem \ref{thm:MainCommutative}. In fact we only need to take care on the order of factors and use formulas
\begin{align}
 \label{eq:BCH-CHLBis}
 \Exp{\lambda u}\Exp {\RotaOp\Bra{\CHL(u)}} &= \Exp{-\RotaOpT\Bra{\CHL(u)}}, \\
 \label{eq:BCH-CHLTercio}
 \Exp{\lambda u} &= \Exp{-\RotaOpT\Bra{\CHL(u)}}\Exp {-\RotaOp\Bra{\CHL(u)}}, \qquad u\in\Algebra_1,
\end{align}
when needed. Note that they both comes from the equality \eqref{eq:BCH-CHL}. Let us sketch the proof briefly.
As previously, we assume $u = \lambda^{-1}\log(1+\lambda a_1)$, so that $a_1 = \lambda^{-1}(\exp(\lambda u)-1)$ and we want to show
\begin{align}
\label{eq:ToProveNonCom}
 \lambda \Exp{\RotaOp(\CHL(u))}\RotaOp\left(\Exp{-\RotaOp(\CHL(u))}a_0\right) = 
 \lambda\RotaOp(\Exp{\lambda u}a_0) + \RotaOp\left( \left(\Exp{\lambda u}-1\right) \Exp{\RotaOp(\CHL(u))}\RotaOp\left(\Exp{-\RotaOp(\CHL(u))}a_0\right) 
  \right).
\end{align}
Using  \eqref{eq:BCH-CHLBis} we get
\begin{align*}
  \left(\Exp{\lambda u}-1\right) \Exp{\RotaOp(\CHL(u))}\RotaOp\left(\Exp{-\RotaOp(\CHL(u))}a_0\right) & =
  \left(\Exp{-\RotaOpT(\CHL(u))} - \Exp{\RotaOp(\CHL(u))}\right) \RotaOp\left(\Exp{-\RotaOp(\CHL(u))}a_0\right).
\end{align*}
Now we do the same as in the previous proof changing $\RotaOp(u)$ and $\RotaOpT(u)$ into $\RotaOp(\CHL(u))$ and $\RotaOpT(\CHL(u))$, respectively, until we obtain an analog of the equation \eqref{eq:Almost}
\begin{multline*}
  \RotaOp\left( \left(\Exp{\lambda u}-1\right) \Exp{\RotaOp(\CHL(u))}\RotaOp\left(\Exp{-\RotaOp(\CHL(u))}a_0\right) \right)  
  = \\
  \lambda \Bra{\Exp{\RotaOpBra{\CHL(u)}} - 1} \RotaOpBra{\Exp{-\RotaOpBra{ \CHL(u)}}a_0} 
  - \lambda\RotaOp\left(\left(\Exp{-\RotaOpT(\CHL(u))} - 1\right)    \Exp{-\RotaOp(\CHL(u))}a_0  \right).
\end{multline*}
Now, using \eqref{eq:BCH-CHLTercio} in the second summand and simplifying the right side we get
\begin{align*}
 \RotaOp\left( \left(\Exp{\lambda u}-1\right) \Exp{\RotaOp(\CHL(u))}\RotaOp\left(\Exp{-\RotaOp(\CHL(u))}a_0\right) \right) =
 \lambda\Bra{\Exp{\RotaOpBra{\CHL (u)}}\RotaOpBra{\Exp{-\RotaOpBra{ \CHL(u)}}a_0} - \RotaOpBra{\Exp{\lambda u}a_0}}.
\end{align*}
So we obtain \eqref{eq:ToProveNonCom}. 
Now, as previously it is easy to see that
\begin{align*}
 \sum_{n=0}^\infty \underbrace{\RotaOp(a_1\cdots\RotaOp(a_1}_{n}\RotaOp((1 + \lambda a_1)a_0))\cdots) 
\end{align*}
is a solution of \eqref{eq:LinearInhomogenEqNonCom},  and since $n$-th component of this sum is in $\Algebra_{n+1}$, the series is convergent. From the uniqueness of the solution we conclude \eqref{eq:TheEqualityNonCom}.

By repeated use of \eqref{eq:LinearInhomogenEqNonCom} we conclude that if $\lambda$ is not a zero divisor of $\Algebra$, then the generalized non-commutative Spitzer's identity \eqref{eq:TheEqualityNonCom} holds true.

The proof of $(ii)$ is quite the same and we omit it.
\end{proof}

\section{Eulerian identities}
\label{sec:Eulerian}

In this section we prove Proposition \ref{prop:Eulerian} and Proposition \ref{prop:EulerianTwo}.
Let $1\neq q\in\QQ$. Consider  the algebra of formal power series in variable $t$ with rational coefficients $\QQ[[t]]$, and  the $q$-integral $\QInt : \QQ[[t]] \to \QQ[[t]]$ given by $\QInt(f)(t) = \sum_{k=1}^\infty f(q^kt)$ so that $\QInt(t^n)(t) = \frac{q^nt^n}{1-q^n}$. It is known that $(\QQ[[t]],\QInt)$ is a commutative Rota-Baxter algebra of weight $1$ (see \cite[Example 1.1.8]{guo2012introduction}). Take $\QQ[[t]]_n$ as power series of order not less than  $n$, i.e., 
$$
\QQ[[t]]_n = \SetSuchThat{\sum_{k=n}^\infty a_k t^k}{ a_i\in\QQ,\, i = n,n+1,\ldots}.
$$ 
Then $(\QQ[[t]],\QQ[[t]]_n ,\QInt)$ is a commutative complete filtered Rota-Baxter algebra of weight $1$.
We will use the following Eulerian identity 
\begin{align}
\label{eq:EuerianIdentity}
 1 + \sum_{n=1}^\infty \frac{q^{\frac{n(n+1)}{2}-1}}{(1 -q)\cdots(1-q^{n})}t^{n} = \prod_{n=1}^\infty \Bra{1+q^nt},
\end{align}
which is Spitzer's identity \eqref{eq:Spitzer} in this algebra applied for $a = t$ (see \cite[Example 1.3.7]{guo2012introduction}). In particular, it will be important that
\begin{align}
\label{eq:EulerianSecond}
 \exp \Bra{\QInt\left(\log(1+ t) \right) } = 
 \prod_{n=1}^\infty \Bra{1+q^nt}.
\end{align}

\begin{proof}[Proof of Proposition \ref{prop:Eulerian}.]
Let us compute the outcome of \eqref{eq:TheEquality} in $(\QQ[[t]],\QQ[[t]]_n,\QInt)$ for $a_0 = a_1 = t$.
The left side of \eqref{eq:TheEquality}. 
First $\QInt((1 + a_1)a_0) = \QInt(t + t^2) = \frac{qt}{1-q} + \frac{q^2t^2}{1-q^2}$. Using induction on $n$ it is easy to see that 
\begin{align*}
\underbrace{\QInt(a_1\cdots\QInt(a_1}_{n}\QInt((1 + \lambda a_1)a_0))\cdots) = 
 \frac{q^{\frac{(n+1)(n+2)}{2}}}{(1 -q)\cdots(1-q^{n+1})}t^{n+1} + \frac{q^{\frac{(n+2)(n+3)}{2}-1}}{(1 -q^2)\cdots(1-q^{n+2})}t^{n+2}.
\end{align*}
Summing up these expressions, we get
\begin{align}
\label{eq:EulerianFirst}
  \sum_{n=0}^\infty \underbrace{\QInt(a_1\cdots\QInt(a_1}_{n}\QInt((1 + \lambda a_1)a_0))\cdots) & 
  = \frac q {1-q} t + \sum_{n=1}^\infty \frac{q^{\frac{(n+1)(n+2)}{2}}}{(1 -q)\cdots(1-q^{n+1})}\Bra{1 + \frac{1-q} q}t^{n+1}
  \nonumber \\
  & = \frac q {1-q} t + \sum_{n=2}^\infty \frac{q^{\frac{n(n+1)}{2}-1}}{(1 -q)\cdots(1-q^{n})}t^{n}
  \nonumber \\
  & = - (1+ t) + 1 + \sum_{n=1}^\infty \frac{q^{\frac{n(n+1)}{2}-1}}{(1 -q)\cdots(1-q^{n})}t^{n}
  \nonumber \\
  & = - (1+t) + \prod_{n=1}^\infty \Bra{1+q^nt},
\end{align}
where the last equality follows from an Eulerian identity \eqref{eq:EuerianIdentity}.

The right side of \eqref{eq:TheEquality}.
The first factor is just the formula \eqref{eq:EulerianSecond}.
According to the second factor, lets compute first
\begin{align}
\label{eq:ComputationOne}
\begin{split}
 \exp \Bra{-\QInt(\log(1+t))} 
 & = \exp \Bra{ \QInt\Bra{\sum_{n=1}^\infty \frac{(-1)^n} n t^n}}
  = \exp \Bra{\sum_{n=1}^\infty \frac{(-t)^n} n \frac {q^n} {1-q^n}}
 \\ & = \exp \Bra{\sum_{n=1}^\infty \frac{(-t)^n} n \sum_{k=1}^\infty q^{nk}}
  = \exp \Bra{\sum_{k=1}^\infty \sum_{n=1}^\infty \frac{(-q^kt)^n} n }
 \\ & = \prod_{k=1}^\infty \exp \Bra{ -\log \Bra{1 + q^kt}}
  = \prod_{k=1}^\infty \frac 1 {1 + q^kt}.
 \end{split}
 \end{align}
 In order to compute $\QInt\Bra{\exp \Bra{-\QInt(\log(1+t))}t}$ we need the following lemma (it roughly follows from the unproven Example 1.3.8 in  \cite{guo2012introduction}).

 \begin{lemma}It follows that
  $$\prod_{k=1}^\infty \frac 1 {1 + q^kt} = (1+t)\Bra{1 + \sum_{n=1}^\infty \frac{(-t)^n}{(1 -q)\cdots(1-q^{n})}}$$
 \end{lemma}
 
\begin{proof}
 Let a linear operator $\QIntTwo : \QQ[[t]] \to \QQ[[t]]$ be given on the homogeneous polynomials by $\QIntTwo(t^n)(t) = \frac{t^n}{1-q^n}$. It is known that $(\QQ[[t]],\QQ[[t]]_n ,\QIntTwo)$ is a Rota-Baxter algebra of weight $-1$ (see \cite[Example 1.1.9]{guo2012introduction}) and it is easy to see that it is complete filtered. In this algebra we use Spitzer's identity \eqref{eq:Spitzer} for $a = t$. Proceeding like in \eqref{eq:ComputationOne} we obtain
 \begin{align*}
  \exp \Bra{\QIntTwo\Bra{-\log(1-t)}} = \prod_{k=0}^\infty \frac 1 {1 - q^kt}.
 \end{align*}
By induction on $n$ we also get that 
$$
\underbrace{\QIntTwo(a_1\cdots\QIntTwo(a_1}_{n}\QIntTwo((1 + \lambda a_1)a_0))\cdots) = \frac{t^n}{(1-q)\cdots(1-q^n)}.
$$ 
From \eqref{eq:Spitzer} we therefore have
\begin{align*}
\frac 1 {1 - t} \prod_{k=1}^\infty \frac 1 {1 - q^kt} = 1 + \sum_{n=1}^\infty \frac{t^n}{(1-q)\cdots(1-q^n)}.
\end{align*}
Taking $-t$ instead of $t$ we obtain the formula.
\end{proof}

We continue the proof of Proposition \ref{prop:Eulerian}.
Using this lemma and equality \eqref{eq:ComputationOne} we see that
\begin{align}
\label{eq:EulerianThird}
\QInt\Bra{\exp \Bra{-\QInt(\log(1+t))}t}
 & = \QInt\Bra{t(1+t)\Bra{1 + \sum_{n=1}^\infty \frac{(-t)^n}{(1 -q)\cdots(1-q^{n})}}}
 \nonumber \\ & = \QInt\Bra{t + \sum_{n=0}^{\infty}\frac{(-1)^{n+1}q^{n+1}t^{n+2}}{(1 -q)\cdots(1-q^{n+1})}}
 \nonumber \\ & = \frac {qt}{1-q} + \sum_{n=0}^{\infty}\frac{(-1)^{n+1}q^{2n+3}t^{n+2}}{(1 -q)\cdots(1-q^{n+2})}
 \nonumber \\ & = - \sum_{n=0}^\infty \frac{q^{2n+1}(-t)^{n+1}}{(1 -q)\cdots(1-q^{n+1})}.
\end{align}
Putting \eqref{eq:EulerianSecond}, \eqref{eq:EulerianFirst}, \eqref{eq:EulerianThird} into \eqref{eq:TheEquality} we get
\begin{align*}
 -(1+t) + \prod_{n=1}^\infty \Bra{1+q^nt} =\Bra{- \sum_{n=0}^\infty \frac{q^{2n+1}(-t)^{n+1}}{(1 -q)\cdots(1-q^{n+1})}} \prod_{n=1}^\infty \Bra{1+q^nt}.
\end{align*}
Changing $-t$ to $t$, and after simple transformations we conclude that
\begin{align*}
 1 + \sum_{n=1}^\infty \frac{q^{2n-1}t^{n}}{(1 -q)\cdots(1-q^{n})} = (1-t)\prod_{n=1}^\infty \frac 1 {1-q^nt}.
\end{align*}
\end{proof}

\begin{proof}[Proof of Proposition \ref{prop:EulerianTwo}]
Let us compute \eqref{eq:SpecialEquality} in the R-B algebra $(\QQ[[t]],\QQ[[t]]_n,\QInt)$ for $a_1 = t$. By \eqref{eq:ComputationOne}, Lemma \ref{lem:Iteration}, and a straightforward calculation we see that
\begin{align*}
 \QInt\Bra{\Exp {-\QInt\left(\log(1+ t) \right) } \frac {a_1}{1+ a_1}} 
 & = \QInt\Bra{ \frac t {1+t} \prod_{n=1}^\infty \frac 1 {1+q^nt}} \\
 & = \QInt\Bra{ t \Bra{1 + \sum_{n=1}^\infty \frac{(-t)^n}{(1 -q)\cdots(1-q^{n})}}} \\
 & = - \sum_{n=1}^\infty \frac{q^n(-t)^n}{(1 -q)\cdots(1-q^{n})}.
\end{align*}
Therefore the formula \eqref{eq:SpecialEquality} reads as
\begin{align*}
 1 + \sum_{n=1}^\infty \frac{q^n(-t)^n}{(1 -q)\cdots(1-q^{n})} = \prod_{n=1}^\infty \frac 1 {1+q^nt}.
\end{align*}
Changing $t$ for $-t$ gives the result.
\end{proof}

\section{Concluding remarks}

In this article we proved a generalization of Spitzer's identity in a complete filtered Rota-Baxter algebra of weight $\lambda$. The idea comes from the formula for the solution of non-homogeneous linear differential equation. This suggest that other formulas 
in Rota-Baxter algebras can also be derived if we imitate solutions of other types of differential equations, like Bernoulli equation, Riccati equation, Abel equation, higher order linear equations, etc. 


\section*{Acknowledgements}

The author was partially supported by the Polish Ministry of Research and Higher Education grant NN201 607540, 2011-2014.

\bibliographystyle{acm}
\bibliography{bibliography}

\begin{thebibliography}{10}

\bibitem{atkinson1963some}
{\sc Atkinson, F.}
\newblock {Some aspects of Baxter's functional equation}.
\newblock {\em Journal of Mathematical Analysis and Applications 7}, 1 (1963),
  1--30.

\bibitem{baxter1960analytic}
{\sc Baxter, G.}
\newblock An analytic problem whose solution follows from a simple algebraic
  identity.
\newblock {\em Pacific J. Math 10}, 3 (1960), 731--742.

\bibitem{carinena2007hopf}
{\sc Carinena, J.~F., Ebrahimi-Fard, K., Figueroa, H., and Gracia-Bond, J.~M.}
\newblock {Hopf algebras in dynamical systems theory}.
\newblock {\em International Journal of Geometric Methods in Modern Physics 4},
  04 (2007), 577--646.

\bibitem{cartier1972structure}
{\sc Cartier, P.}
\newblock {On the structure of free Baxter algebras}.
\newblock {\em Advances in Mathematics 9}, 2 (1972), 253--265.

\bibitem{ebrahimi2004spitzer}
{\sc Ebrahimi-Fard, K., Guo, L., and Kreimer, D.}
\newblock {Spitzer's identity and the algebraic Birkhoff decomposition in
  pQFT}.
\newblock {\em Journal of Physics A: Mathematical and General 37}, 45 (2004),
  11037.

\bibitem{ebrahimi2005integrable}
{\sc Ebrahimi-Fard, K., Guo, L., and Kreimer, D.}
\newblock {Integrable renormalization II: the general case}.
\newblock In {\em Annales Henri Poincare\/} (2005), vol.~6, Springer,
  pp.~369--395.

\bibitem{ebrahimi2006birkhoff}
{\sc Ebrahimi-Fard, K., Guo, L., and Manchon, D.}
\newblock {Birkhoff type decompositions and the Baker--Campbell--Hausdorff
  recursion}.
\newblock {\em Communications in mathematical physics 267}, 3 (2006), 821--845.

\bibitem{ebrahimi2009magnus}
{\sc Ebrahimi-Fard, K., and Manchon, D.}
\newblock {A Magnus-and Fer-type formula in dendriform algebras}.
\newblock {\em Foundations of Computational Mathematics 9}, 3 (2009), 295--316.

\bibitem{guo2012introduction}
{\sc Guo, L.}
\newblock {\em {An Introduction to Rota-Baxter Algebra}}, vol.~2.
\newblock International Press, 2012.

\bibitem{kingman1962spitzer}
{\sc Kingman, J.}
\newblock {Spitzer's identity and its use in probability theory}.
\newblock {\em Journal of the London Mathematical Society 1}, 1 (1962),
  309--316.

\bibitem{magnus1954exponential}
{\sc Magnus, W.}
\newblock On the exponential solution of differential equations for a linear
  operator.
\newblock {\em Communications on pure and applied mathematics 7}, 4 (1954),
  649--673.

\bibitem{rota1995baxter}
{\sc Rota, G.-C.}
\newblock {Baxter operators, an introduction, Kung JPS, Gian-Carlo Rota on
  Combinatorics, Introductory Papers and Commentaries, 1995}.

\bibitem{rota1969baxter}
{\sc Rota, G.-C.}
\newblock {Baxter algebras and combinatorial identities. I}.
\newblock {\em Bulletin of the American Mathematical Society 75}, 2 (1969),
  325--329.

\bibitem{rota1969baxterII}
{\sc Rota, G.-C.}
\newblock {Baxter algebras and combinatorial identities. II}.
\newblock {\em Bulletin of the American Mathematical Society 75}, 2 (1969),
  330--334.

\bibitem{rota1972fluctuation}
{\sc Rota, G.-C., and Smith, D.}
\newblock {Fluctuation theory and Baxter algebras}.
\newblock In {\em Symposia Mathematica\/} (1972), vol.~9, pp.~179--201.

\bibitem{spitzer1956combinatorial}
{\sc Spitzer, F.}
\newblock A combinatorial lemma and its application to probability theory.
\newblock {\em Trans. Amer. Math. Soc 82}, 2 (1956), 323--339.

\end{thebibliography}

\end{document}